\documentclass[a4paper,11pt,centertags]{amsart}
\usepackage[latin1]{inputenc}
\usepackage[english]{babel}
\usepackage{cite}
\usepackage{verbatim}
\usepackage{color}
\usepackage{graphicx,picture}
\usepackage{hyperref}
\usepackage{extarrows}
\usepackage{amsmath,amsthm}
\usepackage{comment}

\usepackage[top=3.2cm, bottom=3.2cm, left=2.8cm,right=2.8cm]{geometry}

\newcommand{\diesis}{^\#}
\newtheorem{theo}{Theorem}[section]
\newtheorem{lemma}{Lemma}[section]

\theoremstyle{definition}
\newtheorem{definiz}{Definition}[section]
\newtheorem{rem}{Remark}[section]

\numberwithin{equation}{section}

\newcommand{\R}{\mathbb R}

\newcommand{\de}{\partial}

\DeclareMathOperator{\divergenza}{div}

\begin{document}
\title[Sharp estimates for the Robin Gaussian torsional rigidity ]
{ Sharp estimates for the Gaussian torsional rigidity  with Robin boundary conditions} 
\author[F. Chiacchio]{
  Francesco Chiacchio
}
 \address{
Francesco Chiacchio \\
Universit\`a degli studi di Napoli ``Federico II''\\
Dipartimento di Ma\-te\-ma\-ti\-ca e Applicazioni ``R. Caccioppoli''\\
Complesso di Monte Sant'Angelo, Via Cintia,
80126 Napoli, Italia.
}
\email{francesco.chiacchio@unina.it}
\author[N. Gavitone]{
  Nunzia Gavitone
}
 \address{
Nunzia Gavitone \\
Universit\`a degli studi di Napoli ``Federico II''\\
Dipartimento di Ma\-te\-ma\-ti\-ca e Applicazioni ``R. Caccioppoli''\\
Complesso di Monte Sant'Angelo, Via Cintia,
80126 Napoli, Italia.
}
\email{nunzia.gavitone@unina.it}
\author[C. Nitsch]{
  Carlo Nitsch
}
 \address{
Carlo Nitsch \\
Universit\`a degli studi di Napoli ``Federico II''\\
Dipartimento di Ma\-te\-ma\-ti\-ca e Applicazioni ``R. Caccioppoli''\\
Complesso di Monte Sant'Angelo, Via Cintia,
80126 Napoli, Italia.
}
\email{c.nitsch@unina.it}
\author[C. Trombetti]{
  Cristina Trombetti
}
 \address{
Cristina Trombetti \\
Universit\`a degli studi di Napoli ``Federico II''\\
Dipartimento di Ma\-te\-ma\-ti\-ca e Applicazioni ``R. Caccioppoli''\\
Complesso di Monte Sant'Angelo, Via Cintia,
80126 Napoli, Italia.
}
\email{cristina@unina.it}
\date{\today}
\maketitle
\begin{abstract}
In this paper we provide a comparison result between the solutions to the torsion problem 
for the Hermite operator with Robin boundary 
conditions and the one of a suitable symmetrized  problem.   

\end{abstract}
\vspace{.5cm}

\noindent\textsc{MSC 2020:} 35J25, 35B45 \\
\textsc{Keywords}: Torsional rigidity, Hermite operator, Robin boundary conditions.

\vspace{.5cm}
\section{Introduction}
Let $\Omega $ be a smooth and possibly unbounded domain of $ \R^n$  and let $\nu$ be the unit outer normal to $\partial \Omega$.
 In this paper we consider the following  torsion problem for the Hermite operator with Robin boundary conditions
\begin{equation}
  \label{eq:2_intro}
  \left\{
    \begin{array}{ll}
      -\divergenza\left(\phi(x)\nabla u\right)= \phi(x) 
      &\text{in } 
      \Omega,\\[.2cm] 
      \dfrac{\partial u}{\partial \nu} +\beta u=0
      &\text{on } \de\Omega. 
    \end{array}
  \right.
\end{equation}

Throughout the paper $\beta$ will be a positive parameter and $\phi(x)$ will denote the density of  the normalized Gaussian measure in $\R^n$, that is 
\[
\phi(x)=\dfrac{1}{(2\pi)^{\frac n2}}\exp\left(-\frac{|x|^2}{2} \right).
\]
The interest in the study of the Hermite operator relies on its
 applications in various fields.
Just to mention a few,  it enters in the description of the harmonic oscillator in quantum mechanics (see, e.g., \cite{bf} and the references therein).
It attracts attention from probabilists too. Indeed, as well known,  the Hermite operator  is the generator of the Ornstein-Uhlenbeck semigroup (see, e.g.,\cite{8} and the references therein).

As we will recall in the next section, suitable weighted embedding trace  theorems hold true if $\Omega$ is  sufficiently smooth.
Therefore, classical arguments ensure that  problem \eqref{eq:2_intro} has a unique positive solution $u$. Furthermore, $u$  is a minimizer of the following functional
\begin{equation}
\label{tor_i}
\displaystyle \frac{1}{T_{\phi}(\Omega)}= \inf_{w \in H^1(\Omega,\phi)\setminus \{0\}} \dfrac{\displaystyle\int_{\Omega} |\nabla w|^2 \phi\,dx +\beta \displaystyle\int_{\partial \Omega}w^2 \phi \, d\mathcal H^{n-1}}{\left(\displaystyle \int_{\Omega}w \phi\,dx\right)^2},
\end{equation}
where $H^1(\Omega,\phi)$ is the   weighted Sobolev space 
naturally associated to problem \eqref{eq:2_intro} (see Section 2 for the definitions and properties). Note that, as a straightforward computation  shows, it holds that
\[
T_{\phi}(\Omega)=\| u\|_{L^1(\Omega,\phi)}:=\int_{\Omega} |u| \phi \,dx.
\]
The aim of this paper, is to prove an isoperimetric inequality for $T_{\phi}(\Omega)$ by means of the so-called Gaussian symmetrization.  This will  be achieved by  comparing  the solution to problem \eqref{eq:2_intro} with that to the following one
 \begin{equation}
  \label{eq:3i}
  \left\{
    \begin{array}{ll}
      -\divergenza\left(\phi(x)\nabla v\right)= \phi(x) 
      &\text{in } 
      \Omega^{\diesis}\\[.2cm] 
      \dfrac{\partial v}{\partial \nu} +\beta v=0
      &\text{on } \de\Omega^{\diesis}, 
    \end{array}
  \right.
\end{equation}  
where $\Omega^{\diesis}:=\{(x_{1},x_{2},\cdots ,x_{n})\in \mathbb{R}^{n}\text{ }\colon \text{
}x_{1}>\lambda\}$ and $\lambda$ is such that $$|\Omega|_{\phi}:=\displaystyle \int_\Omega \phi \,dx=|\Omega^{\diesis}|_{\phi}=h(\lambda),$$
with
\begin{equation}
\label{h}
h(\lambda):=\frac{1}{\sqrt{2\pi }}\int_{\lambda }^{+\infty}\,\text{exp}\left( -\dfrac{t^{2}}{2}\right) dt.\end{equation}


Our  main result is the following
 \begin{theo}
 \label{princ}
  Let $\Omega \in \mathcal G$, see Definition \ref{omega}, 
   Let $u$ and $v$ be the solutions to problems \eqref{eq:2_intro} and  \eqref{eq:3i}, respectively. Then the following comparison result holds
 \begin{equation}
 \label{comp}
\| u\|_{L^1(\Omega,\phi)}\le \| v\|_{L^1(\Omega^{\diesis}\!,\phi)}.
\end{equation}
 \end{theo}
In other words,   among all sufficiently  smooth sets of
$\mathbb R^n$, having prescribed Gaussian measure, the half-spaces maximize 
$T_{\phi}(\Omega)$. 
Note that inequality \eqref{comp} provides a sharp and explicit estimate for $\| u\|_{L^1(\Omega,\phi)}$, since it is elementary to derive the exact form of $v$
\begin{equation}
\label{rad}
v(x)=v(x_1)=C(\lambda,\beta)+\displaystyle\int_{\lambda}^{x_1}\exp\left( \frac{{r}^2}{2}\right)\int_{r}^{+\infty} \exp\left( -\frac{t^2}{2}\right) \, dt\,dr,
\end{equation}
where 
\[
C(\lambda,\beta)=\frac{1}{\beta}\exp\left( \frac{\lambda^2}{2}\right)\int_{\lambda}^{+\infty} \exp\left( -\frac{t^2}{2}\right) \, dt.
\]
 Now let us briefly  describe how our result is inserted in the literature. In \cite{ant} and \cite{bg} the authors investigate the analogous issue for the classical Laplace operator. In particular, in \cite{bg}, the authors obtain an isoperimetric inequality for the Robin torsional rigidity  in a wider context, by studying  a family of Faber-Krahn inequalities. 
They prove that the Robin Laplace torsional rigidity is maximum on balls among all bounded and Lipschitz domains,  once the Lebesgue measure is fixed. 
Their proof, unlike the one used for the Dirichlet boundary conditions, does not make use of any symmetrization techniques, rather, it is based on reflection arguments. 
Recently, in  \cite{ant}, see also \cite{acnt},  the authors obtain the same isoperimetric inequality via a ``Talenti type  comparison result''.
Note that the result contained in \cite{ant} are quite surprising. Indeed, as well known, the Talenti's technique is designed  for problems whose solution has level sets that do not touch the boundary of the domain where the problem is posed. A phenomenon that tipically occurs when Robin boundary conditions are imposed.
In this paper, because of the structure of the differential operator we are considering, in place of the more common Schwarz symmetrization we use 
 the Gauss symmetrization. A procedure that transforms a positive function into a new one having as super level sets half-spaces whose Gauss measure is the same of the original function.

 
The structure of the paper is the  following. In Section 2 we fix some notation  and we recall some results that we will use in the paper. The third Section contains the proof of our main result.

\section{Notation and preliminaries}

Let $A$ be any Lebesgue measurable set of $
\mathbb{R}^{n}$.  The Gaussian perimeter of $A$ is 
\begin{equation*}
P_{\phi }(A)=\left\{ 
\begin{array}{cc}
\displaystyle\int_{\partial A}\phi (x)\text{ }d\mathcal{H}^{n-1}(x) & \text{
if }\partial A\text{ is }(n-1)-\text{rectifiable} \\ 
&  \\ 
+\infty  & \text{otherwise,}
\end{array}
\right. 
\end{equation*}
where $d\mathcal H^{n-1}$ denotes the $(n-1)-$dimensional Hausdorff measure in $\R^n$.

While the Gaussian measure of $A$ is given by 
\begin{equation}
|A|_{\phi }=\int_{A}\phi (x)\,dx\in \left[ 0,1\right] .  \label{mis_pes}
\end{equation}
The celebrated Gaussian isoperimetric inequality (see \cite{ST}, \cite{Bo}
and \cite{Eh}) states that among all Lebesgue measurable sets in $\mathbb{R}
^{n},$ with prescribed Gaussian measure, the half-spaces 
minimize the Gaussian perimeter. Furthermore the isoperimetric set is
unique, clearly, up to a rotation with respect to the origin (see \cite{CK}
and \cite{CFMP}).

The isoperimetric function in the Gauss space, $I(s),$ is 
\begin{equation}
\label{Is_f}
I:s\in \left[ 0,1\right] \rightarrow \frac{1}{\sqrt{2\pi }}\exp \left( -
\frac{(h^{-1}(s))^2}{2}\right),
\end{equation}
where $h^{-1}$ is the inverse function of $h$, defined in \eqref{h}. 
Note, indeed, that the Gaussian perimeter of any half-space of Gaussian
measure $s $ is equal to $I(s).$ The isoperimetric
property of the half-spaces can finally be stated as follows.
\begin{theo}
\label{iso} If $\Omega \subset \mathbb{R}^{n}$ is any Lebesgue measurable
set it holds that

\begin{equation*}
P_{\phi }(\Omega )\geq P_{\phi }(\Omega ^{\diesis })
=
I \left( \left\vert
\Omega \right\vert _{\phi }\right) ,
\end{equation*}
where equality holds, if and only if, $\Omega $ is equivalent to an half-space.
\end{theo}

Let $\Omega \subset \mathbb R^n$ be an open connected set. We will denote by $L^{2}(\Omega ,\phi )$ the set of all real
measurable functions defined in $\Omega $ such that
\begin{equation*}
\left\Vert u\right\Vert _{L^{2}(\Omega ,\phi )}^{2}:=\int_{\Omega
}u^{2}(x)\phi (x)dx<+\infty .
\end{equation*}
For our future purposes we need also to introduce the following weighted
Sobolev space
\begin{equation*}
H^{1}(\Omega ,\phi ):=\left\{ u\in W_{\text{loc}}^{1,1}(\Omega
):(u,\left\vert \nabla u\right\vert )\in L^{2}(\Omega ,\phi)\times L^{2}(\Omega
,\phi )\right\} ,
\end{equation*}
endowed with the norm
\begin{equation*}
\left\Vert u\right\Vert _{H^{1}(\Omega ,\phi )}=\left\Vert u\right\Vert
_{L^{2}(\Omega ,\phi )}+\left\Vert \nabla u\right\Vert _{L^{2}(\Omega ,\phi)}.
\end{equation*}
In the sequel of the paper, we need to introduce the following family of sets.
\begin{definiz}
\label{omega}
A Lipschitz domain $\Omega$ of $\mathbb R^n$ is in $\mathcal G$ 
if  $|\Omega|_{\phi} \in (0,1)$
and  the following conditions are fulfilled:
\begin{itemize}
\item[(i)]  $H^{1}(\Omega ,\phi )$ is compactly embedded in 
$L^{2}(\Omega ,\phi)$. 
\item[(ii)]  The trace operator $T$
\begin{equation*}
T:u\in H^{1}(\Omega ,\phi)\rightarrow \left. u\right\vert _{\partial \Omega
} \in L^{2}(\partial \Omega ,\phi),
\end{equation*}
is well defined;

\item[(iii)] The trace operator defined in the previous point is compact from $H^{1}(\Omega ,\phi )$ onto  $L^{2}(\partial \Omega ,\phi)$.

\end{itemize}
In (ii) and (iii) the functional space $L^{2}(\partial \Omega ,\phi )$ is endowed with the
norm
\begin{equation*}
\left\Vert u\right\Vert _{L^{2}(\partial \Omega ,\phi )}^{2}=\int_{\partial
\Omega }u^{2}(x)\phi (x)d\mathcal{H}^{n-1}(x).
\end{equation*}
\end{definiz} 

We stress that $\mathcal G$ is non empty (see for instance Remark 2.1 in \cite{cg}). 

Finally, we recall the following version of  Gronwall's Lemma.
\begin{lemma}\label{gron} Let $\xi(\tau)$ be a continuously differentiable function satisfying, for some constant $C \ge0$ the following differential inequality
\[
\tau \xi'(\tau)\le \xi(\tau)+C \quad \text{for all }\tau \ge \tau_0>0.
\]
Then
\begin{equation}
\xi(\tau)\le \displaystyle \tau\frac{\xi(\tau_0)+C}{\tau_0} -C\quad \text{for all }\tau \ge \tau_0.
\end{equation}
\begin{equation}
\xi'(\tau)\le \displaystyle \frac{\xi(\tau_0)+C}{\tau_0} \quad \text{for all }\tau \ge \tau_0.
\end{equation}
\end{lemma}

\section{Proof of the main result}
In this section,  $u$ and $v$ will denote  the solutions to problems \eqref{eq:2_intro} and  \eqref{eq:3i},  respectively.
In order to prove our isoperimetric inequality for $T_{\phi}(\Omega)$, we need the following auxiliary result which may have independent interest.
\begin{lemma}
\label{confr}
The following inequalities hold true
\begin{equation}
\label{min}
0 \le u_{m}\le v_{m},
\end{equation}
where  \[ u_{m}:=\inf_{\Omega}
\,\,u, \quad  v_{m}:=
\min_{\overline{\Omega^{\diesis}}} \,\,v.\]
\end{lemma}
\begin{proof}
In order to prove the first inequality in \eqref{min}, we use $u^-:=\max\{0,-u\}$ as test function in \eqref{eq:2_intro}, obtaining 
\[
-\int_{\Omega} |\nabla u^-|^2\phi(x) \,dx-\beta \int_{\partial \Omega}(u^-)^2 \phi(x)\, d \mathcal H^{n-1}=\int_{\Omega}u^- \phi(x)\,dx.
\]
Hence $u^-=0$ a.e. in $\Omega$.

Concerning the second inequality in \eqref{min}, we observe that the function $v(x)=v(x_1)$ defined in \eqref{rad} is  increasing. Therefore it  achieves its minimum $v_{m}$  on $\partial \Omega^{\diesis}$. Let $u$ be the solution to the problem \eqref{eq:2_intro}, then 
\begin{gather}
\begin{split}
v_{m}\,P_{\phi}(\Omega^{\diesis})&=\int_{\partial\Omega^{\diesis}}v\,\phi \, d\mathcal H^{n-1}= -\frac{1}{\beta} \int_{\partial\Omega^{\diesis}}\frac{\partial u}{\partial  \nu} \phi \,\mathcal H^{n-1}\\&=-\frac{1}{\beta}\int_{\Omega^{\diesis}}\text{div}\left(\phi \nabla u \right)\,dx=\frac{1}{\beta} |\Omega^{\diesis}|_{\phi}=\frac{1}{\beta}|\Omega|_{\phi} \\&=\int_{\partial \Omega} u \,\phi d\mathcal H^{n-1}\ge u_{m} P_{\phi}(\Omega)\ge u_{m} P_{\phi}(\Omega^{\diesis}),\label{ult}
\end{split}
\end{gather}
where last inequality follows from the weighted  isoperimetric inequality \eqref{iso}. The claim is  hence proven.  
\end{proof}

In the sequel the following notation will be in force. 

For $t \ge0$ we denote by
\begin{equation*}
U_t=\{x \in \Omega \colon u(x)>t\}, \; \partial U_t^{int}=\partial U_t \cap \Omega,\;\partial U_t^{ext}=\partial U_t \cap \partial\Omega,
\end{equation*}
and by 
\begin{equation}
\label{mu}
\mu(t)=|U_t|_{\phi} \,\,\text{ and } \,\,P_u(t)= P_{\phi}(U_t).
\end{equation}
Analogously  if $t \ge0$ we denote by 
\begin{equation}
\label{fii}
V_t=\{x \in \Omega \colon v(x)>t\},\,\, \varphi(t)=|V_t|_{\phi}\,\,\text{ and }\,\,P_v(t)= P_{\phi}(V_t).
\end{equation}
\vspace{.05cm}
\begin{rem}
An immediate consequence of Proposition \ref{confr} is the following inequality
\begin{equation} 
\label{distr}
\mu(t) \le \varphi (t)=|\Omega|_{\phi} \, \quad \forall t \in [0,v_{m}].
\end{equation}
\end{rem}
\vspace{.2cm}
In order to prove our  main results  we need some further lemmata. 
\begin{lemma}
\label{first}
For a.e. $t \ge 0$ it holds
\begin{equation}
\label{1}
 \frac{1}{2\pi}\exp\left( -\left(h^{-1}(\mu(t))\right)^2\right) \le \mu(t) \left((-\mu'(t))+\frac {1}{\beta} \int_{\partial U_t^{ext}}\displaystyle\frac{\phi(x)}{u}\,d \mathcal H^{n-1}\right),
\end{equation}
while 
\begin{equation}
\label{2}
 \frac{1}{2\pi}\exp\left( -\left(h^{-1}(\varphi(t))\right)^2\right) = \varphi(t) \left((-\varphi'(t))+\frac {1}{\beta} \int_{\partial V_t}\frac{\phi(x)}{v}\,d \mathcal H^{n-1}\right),
\end{equation}
where $h$ is defined in \eqref{h}.
\end{lemma}
\begin{proof} Sard's Lemma ensures that $U_t$ is a regular level set, for almost every $t\ge0$. Then it holds
\begin{equation}
\label{conto1}
\begin{split}
P^2_u(t)&=\left(\int_{\partial U_t}\phi(x)\,d \mathcal H^{n-1}\right)^2=\left(\int_{\partial U_t}\left(\frac{\phi(x)}{\big|\frac{\partial u}{\partial \nu}\big|}\right)^{\frac 12}\left(\phi(x)\Big|\frac{\partial u}{\partial \nu}\Big|\right)^{\frac 12}\,d \mathcal H^{n-1}\right)^2\\
&\le \left(\int_{\partial U_t}\frac{\phi(x)}{\big|\frac{\partial u}{\partial \nu}\big|}\,d \mathcal H^{n-1}\right) \left( \int_{\partial U_t}\phi(x)\Big|\frac{\partial u}{\partial \nu}\Big|\,d \mathcal H^{n-1}\right) = \mu(t)\left(\int_{\partial U_t}\frac{\phi(x)}{\big|\frac{\partial u}{\partial \nu}\big|}\,d \mathcal H^{n-1}\right)\\&=\mu(t)\left(\int_{\partial U_t^{int}}\frac{\phi(x)}{\big|\frac{\partial u}{\partial \nu}\big|}\,d \mathcal H^{n-1}+\int_{\partial U_t^{ext}}\frac{\phi(x)}{\big|\frac{\partial u}{\partial \nu}\big|}\,d \mathcal H^{n-1}\right) \\&=\mu(t) \left((-\mu'(t))+\frac {1}{\beta} \int_{\partial U_t^{ext}}\displaystyle\frac{\phi(x)}{u}\,d \mathcal H^{n-1}\right).
\end{split} 
\end{equation}
The Gaussian isoperimetric inequality \eqref{iso} gives 
\begin{equation}
\label{below}
P^2_u(t) \ge \frac{1}{2\pi}\text{exp}\left( -\left(h^{-1}(\mu(t))\right)^2\right),
\end{equation}
where $h$ is defined in \eqref{h}.
Inequalities \eqref{below} and \eqref{conto1} finally imply
\begin{equation}
\label{ta_u}
 \frac{1}{2\pi}\text{exp}\left( -\left(h^{-1}(\mu(t))\right)^2\right) \le \mu(t) \left((-\mu'(t))+\frac {1}{\beta} \int_{\partial U_t^{ext}}\displaystyle\frac{\phi(x)}{u}\,d \mathcal H^{n-1}\right),
\end{equation}
which is  inequality \eqref{1}. Clearly, repeating the same arguments  for the function $v$, we get equality  \eqref{2} in place     of  inequality \eqref{ta_u}.
\end{proof}
The following result allows to handle the right hand side in \eqref{ta_u}.
\begin{lemma}
\label{fub}
Let $v_m$ be the minimum of $v$. For almost every  $t \ge v_m$ it holds
\begin{equation}
\label{fu}
\displaystyle\int_0^{\tau} t \left(\int_{\partial U_t^{ext} }\displaystyle\frac{\phi(x)}{u(x)}\,d \mathcal H^{n-1}(x)\right)\,dt\le\frac{|\Omega|_{\phi}}{2\beta},
\end{equation}
and
\begin{equation}
\label{fv}
\displaystyle\int_0^{\tau} t \left(\int_{\partial V_t \cap \partial \Omega^{\diesis} }\displaystyle\frac{\phi(x)}{v(x)}\,d \mathcal H^{n-1}(x)\right)\,dt=\frac{|\Omega|_{\phi}}{2\beta}.
\end{equation}
\end{lemma}
\begin{proof}
Fubini's Theorem yields
\begin{equation}
\begin{split}
\label{fub1}
\displaystyle\int_{\partial \Omega }\phi(x)\left( \displaystyle \int_0^{u(x)}\frac{t}{u(x)}\,dt\right)\,d \mathcal H^{n-1}(x)&=\displaystyle\int_{\partial \Omega }\phi(x)\left( \displaystyle \int_0^{\infty}\frac{t}{u(x)} \chi_{\{u>t\}}\,dt\right)\,d \mathcal H^{n-1}(x)\\
&=\int_0^{\infty} t \left(\displaystyle \int_{\partial U_t^{ext}} \frac{\phi(x)}{u(x)}\,d \mathcal H^{n-1}(x)\right)\,dt,
\end{split}
\end{equation}
where $\chi$ stands for the characteristic function.
Since $u$ is the solution to  problem \eqref{eq:2_intro} it holds
\begin{equation}
\begin{split}
\label{fub2}
\displaystyle\int_{\partial \Omega }\phi(x)\left( \displaystyle \int_0^{u(x)}\frac{t}{u(x)}\,dt\right)\,d \mathcal H^{n-1}(x)=&\frac12 \displaystyle\int_{\partial \Omega }\phi(x) u(x)\,d \mathcal H^{n-1}(x)=\\&-\frac{1}{2\beta} \displaystyle\int_{\partial \Omega }\phi(x) \frac{\partial u}{\partial \nu}\,d \mathcal H^{n-1}(x)=\frac{|\Omega|_{\phi}}{2\beta}.
\end{split}
\end{equation}
Observing that \[\int_0^{\tau} t \left(\displaystyle \int_{\partial U_t^{ext}} \frac{\phi(x)}{u(x)}\,d \mathcal H^{n-1}(x)\right)\,dt\le\int_0^{\infty} t \left(\displaystyle \int_{\partial U_t^{ext}} \frac{\phi(x)}{u(x)}\,d \mathcal H^{n-1}(x)\right)\,dt,\]
from \eqref{fub} and \eqref{fub2} we get \eqref{fu}.
On the other hand, by repeating the same arguments, we get  \eqref{fv}. Note that,  for all $\tau \ge v_m$ the following equality holds true
\[ \int_0^{\tau} t \left(\displaystyle \int_{\partial V_t\cap \partial \Omega^{\diesis}} \frac{\phi(x)}{v(x)}\,d \mathcal H^{n-1}(x)\right)\,dt =\int_0^{\infty} t \left(\displaystyle \int_{\partial V_t\cap \partial \Omega^{\diesis}} \frac{\phi(x)}{v(x)}\,d \mathcal H^{n-1}(x)\right)\,dt=\frac{|\Omega|_{\phi}}{2\beta},\]  since $\partial V_t\cap \partial \Omega^{\diesis}=\emptyset$ for any $\tau >v_m$.
\end{proof}

Now we can prove our main result.
\begin{proof}[Proof of Theorem \ref{princ}]
We first observe that the function 
\begin{equation}
\label{f}
\text{\ }F(s):=\frac{\text{exp}\left( -\left( h^{-1}(s)\right) ^{2}\right) }{
s^{2}}
\end{equation}
is strictly decreasing  $\forall s\in (0,1)$, where $h$ is the function defined in \eqref{h}.
More precisely we are going to show that 
\begin{equation}
\label{derf}
F'(s)<0\quad \forall s \in (0,1).
\end{equation}
A straightforward computation gives
\begin{equation*}
F^{\prime }(s) =-2\frac{\text{exp}\left( -\left( h^{-1}(s)\right) ^{2}\right) }{s^{3}}+
\frac{1}{s^{2}}\left[ -2\frac{h^{-1}(s)}{h^{\prime }(h^{-1}(s))}\right] 
\text{exp}\left( -\left( h^{-1}(s)\right) ^{2}\right) 
\end{equation*}
\begin{equation*}
=-2\frac{\text{exp}\left( -\left( h^{-1}(s)\right) ^{2}\right) }{s^{3}}+
\frac{1}{s^{2}}\left[ -2\frac{h^{-1}(s)}{-\frac{1}{\sqrt{2\pi }}\text{exp}
\left( -\frac{(h^{-1}(s))^{2}}{2}\right) }\right] \text{exp}\left( -\left(
h^{-1}(s)\right) ^{2}\right) 
\end{equation*}
\begin{equation*}
=-\frac{2}{s^{2}}\frac{\text{exp}\left( -\left( h^{-1}(s)\right) ^{2}\right) 
}{s}+\frac{2}{s^{2}}\sqrt{2\pi }h^{-1}(s)\text{exp}\left( -\frac{
(h^{-1}(s))^{2}}{2}\right) 
\end{equation*}
\begin{equation*}
=-\frac{2}{s^{2}}\left[ \frac{\text{exp}\left( -\left( h^{-1}(s)\right)
^{2}\right) }{s}-\sqrt{2\pi }h^{-1}(s)\text{exp}\left( -\frac{\left(
h^{-1}(s)\right) ^{2}}{2}\right) \right].
\end{equation*}
Therefore $F'(s)<0$ if and only if
\begin{equation*}
\frac{1}{s}\text{exp}\left( -\frac{\left( h^{-1}(s)\right)
^{2}}{2}\right) -\sqrt{2\pi }h^{-1}(s)>0, \quad \forall s \in (0,1).
\end{equation*}
Setting $t:=h^{-1}(s)$, the last inequality is equivalent to the following one
\begin{equation*}
\Psi (t):=\text{exp}\left( -\frac{t^{2}}{2}\right)
-t\int_{t}^{+\infty }\text{exp}\left( -\frac{\sigma ^{2}}{2}\right) d\sigma
>0, \quad \forall t \in \mathbb R.
\end{equation*}
Clearly it holds that
\begin{equation*}
\Psi (t)>0\text{ \ \ \ }\forall t\in \left( -\infty ,0\right] .
\end{equation*}
On the other hand
\begin{equation*}
\Psi ^{\prime }(t)=-\int_{t}^{+\infty }\text{exp}\left( -\frac{\sigma ^{2}}{2
}\right) d\sigma <0\,\quad\forall t\in \mathbb{R}.
\end{equation*}

Hence we get the claim \eqref{derf} if we show that
\begin{equation*}
\lim_{t\rightarrow +\infty }\Psi (t)=0.
\end{equation*}
This is easily verified since on one hand
\begin{equation*}
\lim_{t\rightarrow +\infty }\text{exp}\left( -\frac{t^{2}}{2}\right) =0
\end{equation*}
on the other L'H$\hat{\text o}$pital's rule ensures that
\begin{equation*}
\lim_{t\rightarrow +\infty }t\int_{t}^{+\infty }\text{exp}\left( -\frac{
\sigma ^{2}}{2}\right) d\sigma =\lim_{t\rightarrow
+\infty }\frac{t^{2}}{\text{exp}\left( \frac{t^{2}}{2}\right) }=0.
\end{equation*}
In order to prove \eqref{comp}, 
we first multiply  each side of inequality \eqref{1}  by $t\,\mu(t)\,\text{exp} \left( \left(h^{-1}(\mu(t))\right)^2\right)$ obtaining 
\begin{multline}
\label{3}
\frac{1}{2\pi}\displaystyle t \,\mu(t) \, 
\le \displaystyle  t\,(\mu(t))^2\, \text{exp}\left( \left(h^{-1}(\mu(t))\right)^2\right)(-\mu'(t))+\\+\displaystyle \frac {1}{\beta} t\,(\mu(t))^2\,\text{exp}\left( \left(h^{-1}(\mu(t))\right)^2\right)\int_{\partial U_t^{ext}}\displaystyle\frac{\phi(x)}{u}\,d \mathcal H^{n-1}. 
\end{multline}
We then  integrate between $0$ and $\tau$ such inequality,   obtaining, $\forall \tau \ge v_m$,
\begin{multline}
\label{3}
\frac{1}{2\pi}\displaystyle \int_0^{\tau}t \,\mu(t) \,dt 
\le \displaystyle \int_0^{\tau} t\,(\mu(t))^2\, \text{exp}\left( \left(h^{-1}(\mu(t))\right)^2\right)(-\mu'(t))\,dt+\\+\displaystyle \frac {1}{\beta}\int_0^{\tau} t\,(\mu(t))^2\,\text{exp}\left( \left(h^{-1}(\mu(t))\right)^2\right)\int_{\partial U_t^{ext}}\displaystyle\frac{\phi(x)}{u}\,d \mathcal H^{n-1}\,dt. 
\end{multline}
Note that inequality \eqref{derf} ensures that the function
\[
s^2 \text{exp}\left( \left(h^{-1}(s)\right)^2\right)=\frac{1}{F(s)},
\]
is strictly increasing in $(0,1)$. Therefore inequality \eqref{3} together with Lemma \ref{fub}, implies
\begin{multline}
\label{4}
\frac{1}{2\pi}\displaystyle \int_0^{\tau}t \,\mu(t) \,dt  \le \displaystyle \int_0^{\tau} t\,(\mu(t))^2\, \text{exp}\left( \left(h^{-1}(\mu(t))\right)^2\right)(-\mu'(t))\,dt+\\+\displaystyle \frac {\text{exp}\left( \left(h^{-1}(|\Omega|_{\phi})\right)^2\right)|\Omega|^3_{\phi}}{2\beta^2}, \,\,\forall \tau \ge v_m. 
\end{multline}
Let us define the following function
\[
H(l)=\int_0^l\,s^2 \text{exp}\left( \left(h^{-1}(s)\right)^2\right)\,ds.
\]
Integrating by parts  both sides in inequality \eqref{4}, we get
\begin{multline}
\label{5}
\tau\int_0^{\tau}\frac{\mu(t)}{2\pi}\,dt+\tau H\left(\mu(\tau)\right)\,dt \le \displaystyle \int_0^{\tau} \int_0^t\frac{\mu(r)}{2\pi}\,dr\,dt+ \displaystyle \int_0^{\tau} H\left( \mu(t)\right)\,dt+\\+\displaystyle \frac {\text{exp}\left( \left(h^{-1}(|\Omega|_{\phi})\right)^2\right)|\Omega|^3_{\phi}}{2\beta^2}, \, \quad \forall \tau \ge v_m. 
\end{multline}
Lemma \ref{gron}, ensures that,  $\forall \tau \ge v_m$ it holds
\begin{equation}
\label{6}
\int_0^{\tau}\frac{\mu(t)}{2\pi}\,dt+H\left(\mu(\tau)\right)\,dt  \le \displaystyle \frac{1}{v_m}\left\{ \displaystyle \int_0^{v_m} \int_0^t\frac{\mu(r)}{2\pi}\,dr\,dt+ \displaystyle \int_0^{v_m} H\left( \mu(t)\right)\,dt+\displaystyle \frac {\text{exp}\left( \left(h^{-1}(|\Omega|_{\phi})\right)^2\right)|\Omega|^3_{\phi}}{2\beta^2}
\right\}. 
\end{equation}
Repeating the same procedure for the solution to the problem \eqref{eq:3i}, we obtain the following equality
\begin{equation}
\label{t}
\int_0^{\tau}\frac{\varphi(t)}{2\pi}\,dt+H\left(\varphi(\tau)\right)\,dt  = \displaystyle \frac{1}{v_m}\left\{ \displaystyle \int_0^{v_m} \int_0^t\frac{\varphi(r)}{2\pi}\,dr\,dt+ \displaystyle \int_0^{v_m} H\left( \varphi(t)\right)\,dt+\displaystyle \frac {\text{exp}\left( \left(h^{-1}(|\Omega|_{\phi})\right)^2\right)|\Omega|^3_{\phi}}{2\beta^2}
\right\}. 
\end{equation}
Taking into account of \eqref{distr}, we can compute the right hand side in \eqref{t}, as follows
\begin{equation}
\label{fi2}
\int_0^{\tau}\frac{\varphi(t)}{2\pi}\,dt+H\left(\varphi(\tau)\right)\,dt  = \frac{v_m |\Omega|_{\phi}}{4\pi}+ H(|\Omega|_{\phi})+\displaystyle \frac {\text{exp}\left( \left(h^{-1}(|\Omega|_{\phi})\right)^2\right)|\Omega|^3_{\phi}}{2v_m\beta^2}.
\end{equation}
Then we can compare \eqref{5} and \eqref{6} obtaining 
\begin{equation}
\label{7}
\int_0^{\tau}\frac{\mu(t)}{2\pi}\,dt+H\left(\mu(\tau)\right)\,  \le \int_0^{\tau}\frac{\varphi(t)}{2\pi}\,dt+H\left(\varphi(\tau)\right)\,, \quad \forall \tau \ge v_m.
\end{equation}
Passing to the limit for $\tau \to +\infty$ we get the claim.
\end{proof}
\vspace{.1cm}
\vspace{-.3cm}

\section*{Statements and Declarations}
\begin{itemize}\item[i)] \textbf{Founding:} this work has been partially supported by the PRIN project 2017JPCAPN (Italy) grant: 
\lq\lq Qualitative and quantitative aspects of nonlinear PDEs\rq\rq,by FRA 2020 \lq\lq Optimization problems in Geometric-functional inequalities and nonlinear PDE's \rq\rq (OPtImIzE) and by GNAMPA of INdAM.\\
\item [ii)]\textbf{Competing interests:} 
on behalf of all authors, the corresponding author declares that there are no  financial or non-financial interests that are directly or indirectly related to the work submitted for publication.\\
\item[iii)] \textbf{Availability of data and material:} not applicable.\\	
\item[iv)] \textbf{Code availability:} not applicable.\\	%
\end{itemize}


\begin{thebibliography}{10}

\bibitem{ant}
A.~Alvino, C.~Nitsch and C.~Trombetti. 
\newblock {A Talenti comparison result for solutions to elliptic problems with Robin boundary conditions}, arXiv:1909.11950, to appear on Comm.on Pure and Applied Math.  

\bibitem{acnt}
A.~Alvino, F.~Chiacchio, C.~Nitsch and C.~Trombetti. 
\newblock {Sharp estimates for solutions to elliptic problems with mixed boundary conditions}. J. Math. Pures Appl. (9) 152 (2021), 251?261.

%
\bibitem{8}
V.I.~Bogachev. 
\newblock {Gaussian Measures}.
\newblock {\em Mathematical Surveys and Monographs}, vol. 62, American Mathematical Society, Providence, RI, 1998.

\bibitem{bf}
F.W.~Byron and R.W.~Fuller. 
\newblock{Mathematics of classical and quantum physics.} 
\newblock {\em Dover Publications}, Inc., New York, 1992.

\bibitem{Bo} C.~Borell. 
\newblock {The Brunn-Minkowski inequality in Gauss space}.
\newblock {\em Invent. Math.} 30, 207--216, 1975.
%

\bibitem{bg} D.~Bucur and A.~Giacomini.
\newblock {The Saint-Venant inequality for the Laplace operator with Robin boundary conditions}.
\newblock {\em  Milan J. Math. }, 83 , no. 2, 327--343, 2015.

\bibitem{CK} E. A.~Carlen and C.~Kerce.
 \newblock { On the cases of equality in Bobkov's
inequality and Gaussian rearrangement},
\newblock {\em Calc. Var. Partial Differential },
Equations 13,1-18,  2001.

\bibitem{CFMP} A.~Cianchi, N.~Fusco, F.~Maggi and A.~Pratelli, 
\newblock{On the isoperimetric deficit in Gauss space},
\newblock {\em Amer. J. Math.}, 133, no. 1,
131--186,  2011.

\bibitem{cg} F.~Chiacchio and N.~Gavitone, 
\newblock{The Faber-Krahn inequality for the Hermite operator with Robin boundary conditions}, Math Annalen,  in press.

\bibitem{Eh}
A.~Ehrhard.
\newblock {El\'ements extremaux pour les in\'egalit\'es de Brunn-Minkowski gaussennes}.  \newblock {\em Ann. Inst. H. Poincar\'e Anal. Non Lin\'eaire}, 
22, 149--168, 1986.

%
%
%
%
%

\bibitem{ST} V.~N. Sudakov and B.~S. Cirel'son.
\newblock {Extremal properties of
half-spaces for spherically invariant measures, Extremal properties of
half-spaces for spherically invariant measures (Russian)}. 
Problems in theTheory of Probability Distributions, II, Zap. Nau\v{c}n. Sem. Leningrad.
Otdel. Mat. Inst. Steklov. (LOMI) 41 (1974), 14--24, 165.

\end{thebibliography}
\end{document}